\documentclass{article}

\usepackage{amsthm}
\newtheorem{theorem}{Theorem}[section]
\theoremstyle{definition}

\newcommand{\qbinom}[2]{\binom{#1}{#2}_{\!q}}

\title{Bowling ball representations of braid groups}
\author{Stephen Bigelow}
\date{}

\usepackage{amsfonts}
\usepackage{amsmath}

\begin{document}

\maketitle

\begin{abstract}
In a remark in his seminal 1987 paper,
Jones describes a way to define the Burau matrix of a positive braid
using a metaphor of bowling a ball
down a bowling alley with braided lanes.
We extend this definition
to allow multiple bowling balls to be bowled simultaneously.
We obtain the Iwahori-Hecke algebra
and a cabled version of the Temperley-Lieb representation.
\end{abstract}

\section{Introduction}

The positive braid monoid $B_n^+$ is the monoid
generated by $\sigma_1, \dots, \sigma_{n-1}$
modulo the following relations.
\begin{description}
\item[Far commutativity:]
     $\sigma_i \sigma_j = \sigma_j \sigma_i$ if $|i-j| > 1$.
\item[The braid relation:]
     $\sigma_i \sigma_j \sigma_i = \sigma_j \sigma_i \sigma_j$
     if $|i - j| = 1$.
\end{description}
Alternatively,
$B_n^+$ is the set of geometric braids with $n$ strands
that involve only positive crossings.

In a remark in \cite{jones},
Jones describes a definition of
the (non-reduced) Burau representation of the positive braid monoid
using a ``bowling ball'' metaphor.
Here is the relevant passage
(except we change ``$t$'' to ``$1-q$'' and $(i,j)$ to $(j,i)$,
to match our conventions).

\begin{quote}
For positive braids there is also a mechanical interpretation of the Burau
matrix: Lay the braid out flat and make it into a bowling alley with $n$ lanes,
the lanes going over each other according to the braid. If a ball traveling
along a lane has probability $1 - q$ of falling off the top lane (and continuing
in the lane below) at every crossing then the $(j, i)$ entry of the (non-reduced)
Burau matrix is the probability that a ball bowled in the $i$th lane will end
up in the $j$th.
\end{quote}

This approach was generalized to string links in \cite{lintianwang}.
Subsequent papers,
for example
\cite{klw}, \cite{linwang}, and \cite{armond},
have pursued the related idea of
random walks on braids and knot diagram.
Our goal is to generalize the bowling ball definition
to allow several balls to be bowled simultaneously.
We obtain the Iwahori-Hecke algebra
and a cabled version of the Temperley-Lieb representation.

Throughout the paper,
we work over an arbitrary field
containing an element $q$.
The probability metaphor only makes literal sense
when $q$ is a real number in the range $[0,1]$.
However the results are true for any value of $q$,
and even over a ring.

\section{Definition of the representation}

Fix $N \ge 1$.
Let $\beta$ be a positive braid,
thought of as a bowling alley with $n$ lanes.
Simultaneously bowl balls into the lanes
so that each lane receives at most $N$ balls.

At each crossing,
some balls may fall,
according to the following rule.
Suppose $a$ balls arrive on the top lane of a crossing,
and $b$ arrive on a bottom.
If $a \le b$ then no balls will fall.
If $a > b$ then,
with probability $1 - q$,
exactly $a - b$ balls will fall from the top lane
to join the $b$ balls on the lane below.

Use this to define a matrix $\rho(\beta)$
whose rows and columns are indexed
by $n$-tuples $\mathbf{u} = (u_1, \dots, u_n)$
of integers such that $0 \le u_i \le N$.
The $(\mathbf{v},\mathbf{u})$ entry of $\rho(\beta)$
is the probability that,
if $u_i$ balls are bowled into the $i$th lane for all $i$,
then $v_j$ balls end up in the $j$th lane for all $j$.

\begin{theorem}
$\rho$ is a well-defined representation of $B_n^+$.
\end{theorem}

\begin{proof}
The definition of $\rho$ clearly respects
multiplication,
and the far commutativity relation.
It remains to check the braid relation.

For convenience,
we can assume $n = 3$.
Let
$$\beta = 
\sigma_1 \sigma_2 \sigma_1 = 
\sigma_2 \sigma_1 \sigma_2.$$
Call the three lanes the top, middle, and bottom.
The top lane crosses over both other lanes,
and the bottom lane crosses under both other lanes.

We will compute the entries of the matrix $\rho(\beta)$
in a way that does not depend on the specific word
used to represent $\beta$.
This will show that $\rho(\beta)$ is well-defined.

\textbf{Case 0: No balls.}

If no balls are bowled in to $\beta$
then no balls will emerge.

\textbf{Case 1: One ball.}

Suppose one ball is bowled into $\beta$.
If it is bowled into one of the lower two lanes
then the top lane plays no role,
so we can simply use the probabilities
for a single crossing between the lower two lanes.

Now suppose the ball is bowled into the top lane.
The probability that it will end up in the top lane is $q^2$,
since it must pass over two empty lanes.
The probability that it will end up among
the top two lanes is $q$,
since it must pass over the bottom lane exactly once,
regardless of whether or not it falls
to the middle lane.
By subtraction,
the probability that it will end up 
in the middle lane is $q - q^2$,
and the probability that it will end up 
in the bottom lane is $1 - q$.

\textbf{Case 2: Two balls.}

Suppose two balls are bowled into $\beta$.
If they are bowled into the same lane
then they behave as a single ball,
which was covered in the Case 1.
If one of them is bowled into the bottom lane
then the bottom lane plays no role,
so we can simply use the probabilities
for a single crossing between the upper two lanes.

Now suppose the two balls are bowled into
the top and the middle lanes.
The probability that the bottom lane ends up empty
is $q^2$,
since the two balls must pass over the bottom lane.
The probability that one of the lower two lanes ends up empty
is $q$,
since the top ball must pass over the empty lane.
By subtraction,
the probability that the middle lane will end up empty 
is $q - q^2$,
and the probability that the top lane will end up empty 
is $1 - q$.

\textbf{Case 3: Three balls.}

Suppose three balls are bowled into $\beta$.
If they are bowled into the same lane
then they behave as a single ball,
which was covered in Case 1.
If they are bowled one into each lane
then no balls will fall.

Now suppose one ball is bowled into one lane
and two balls are bowled into another.
We can compute the probability that any given lane ends up empty
by ignoring the distinction between
having one or two balls in a lane,
and proceeding as one ball had been bowled into each of two lanes.
Similarly,
we can compute the probability that any given lane
ends up with two balls
by ignoring the distinction between
having one or zero balls in a lane,
and proceeding as if one ball had been bowled in.

Consider the six possible outcomes in the following order:
$$(0,1,2), (0,2,1), (1,2,0), (2,1,0), (2,0,1), (1,0,2).$$
Any pair of consecutive terms in this list
represents either
the two outcomes that have $0$ in a given position
or the two outcomes that have $2$ in a given position.
This means we have computed the sum of the probabilities of
any two consecutive outcomes in the list.
It remains to compute the probability of
any one of the possible outcomes.

Consider the outcome $(0,1,2)$,
that is,
two balls in the top lane and one in the middle.
If the input is $(2,1,0)$
then the probability of the outcome $(0,1,2)$
is $q^3$,
since there are three crossings
at which a larger number of balls
must pass over a smaller number
without falling.
If the input is any other permutation of $(2,1,0)$
then the probability of the outcome $(0,1,2)$
is $0$,
since balls cannot fall up.
We can now deduce the probability of any outcome
for any given input of three bowling balls.

\textbf{Case 4: The general case.}

Suppose $a$, $b$ and $c$ balls
are bowled into the lanes.
The only thing that matters
about the numbers $a$, $b$ and $c$
is which equalities and inequalities hold between them.
Thus we can reduce to one of the cases we have already covered.

In every case,
for any given input
we can compute the probability of any given output.
Our computation applies equally well to
$\sigma_1 \sigma_2 \sigma_1$ and
$\sigma_2 \sigma_1 \sigma_2$,
so these have the same matrix.
\end{proof}

\section{The Iwahori-Hecke algebra}

Let $\rho$ be the representation of $B_n^+$
defined in the previous section.
The Iwahori-Hecke algebra $H_n(q)$ is the monoid algebra
of formal linear combinations of positive braids
modulo the quadratic relations
$$(q + \sigma_i)(1 - \sigma_i) = 0.$$

\begin{theorem}
$\rho$ factors through $H_n(q)$.
\end{theorem}

\begin{proof}
Only two lanes are involved,
so it suffices to treat the case $n = 2$.
Let $\mathbf{v}$ be the vector corresponding
to $(a, b)$.
We must show that $\mathbf{v}$ is in the kernel of
$$(q + \rho(\sigma_1))(1 - \rho(\sigma_1)).$$

If $a = b$ then $\mathbf{v}$ is fixed by $\rho(\sigma_1)$,
and we are done.
If $a \neq b$
then the action of $\rho(\sigma_1)$ on
the the vectors corresponding to $(a, b)$ and $(b, a)$
is the same as the Burau representation.
It is well known,
and easily checked,
that this satisfies the required quadratic relation.
\end{proof}

Note that if $q$ is invertible
then the $\sigma_i$ are invertible in $H_n(q)$,
with
$$\sigma_i^{-1} = q^{-1}(\sigma_i + q - 1).$$
In this case,
$\rho$ is a representation of the braid group $B_n$
and not just $B_n^+$.

There is another important element of the kernel
in the case $n \ge N + 2$.
Fix $k$ with $1 \le k \le n - N - 1$.
Suppose $w$ is a
permutation of $\{k, \dots, k + N + 1\}$.
Let $\operatorname{sgn}(w) = \pm 1$
denote the sign of the permutation.
Let $\beta_w$ denote the unique positive braid
with a minimal number of crossings
such that the lane at position $i$
goes to position $w(i)$ for all $i = k, \dots, k + N + 1$.
Let $x_k$ be the following element of $H_n(q)$.
$$x_k = \sum_w \operatorname{sgn}(w) \beta_w.$$

A generalization of $x_k$ appears in
the definition of the Specht modules in
\cite{dipperjames}.
The quotient of $H_n(q)$
by the elements $x_k$
is the representation of $B_n$
corresponding to the Lie group
$\operatorname{SL}(N + 1)$,
as defined in \cite{reshetikhinturaev}.
If $N = 1$ then it is the Temperley-Lieb algebra.

We will need the following property of $x_k$.
\begin{equation} \label{eq:factor}
x_k = \left(
\sum_{w(i) < w(i + 1)} \operatorname{sgn}(w) \beta_w
\right)
(1 - \sigma_i),
\end{equation}
for all $i = k, \dots, k + N$.

\begin{theorem}
$\rho(x_k) = 0$.
\end{theorem}

\begin{proof}
It suffices to treat the case $n = N + 2$ and $k = 1$.
Let $x = x_1$.

Suppose we bowl balls into the lanes
so that the number of balls increases
from left to right.
That is,
consider a basis vector $\mathbf{v}$ corresponding to
bowling $v_i$ balls into the $i$th lane,
where
$$0 \le v_1 \le \dots \le v_{N + 2} \le N.$$
Then $v_i = v_{i + 1}$ for some $i$.
Thus $\rho(\sigma_i)$ fixes $\mathbf{v}$.
By Equation \eqref{eq:factor},
$$\rho(x)(\mathbf{v}) = 0.$$

Equation \eqref{eq:factor} also implies
$x \sigma_j = -q x$ for all generators $\sigma_j$.
By induction,
$$\rho(x) \rho(\beta)(\mathbf{v}) = 0$$
for all braids $\beta$.
In particular,
consider $\beta_w$ for a permutation $w$.
If we bowl balls into $\beta_w$
as prescribed by $\mathbf{v}$
then no balls will fall,
since a smaller number passes over a larger number
at every crossing.
The effect is to permute the terms $v_i$.

Every basis vector
is some such permutation of a vector where
the number of balls increases from left to right.
Thus every basis vector is in the kernel of $\rho(x)$.
\end{proof}

\section{A cabling of the Temperley-Lieb algebra}

Fix $K \ge 1$.
Let $\beta$ be a positive braid,
thought of as a bowling alley with $n$ lanes.
Create a {\em cabling} of $\beta$
by replacing every lane with $K$ parallel lanes.
Suppose $\mathbf{a} = (a_1,\dots,a_n)$
is an $n$-tuple of integers such that $0 \le a_i \le K$.
Bowl balls into the lanes so that
each lane gets at most one ball and,
for all $i$,
$a_i$ balls go into the $i$th collection of $K$ parallel lanes.
Whenever a ball passes over an empty lane,
it falls with probability $1 - q$.

Use this to define a matrix $\rho_K(\beta)$
whose rows and columns are indexed
by $n$-tuples $(a_1, \dots, a_n)$
of integers such that $0 \le a_i \le K$.
The $(\mathbf{b},\mathbf{a})$ entry of $\rho_K(\beta)$
is the probability that,
if $a_i$ balls are bowled into the $i$th set of $N$ parallel lanes
for all $i$,
then $b_i$ balls end up in the $i$th set of $N$ parallel lane
for all $i$.

\begin{theorem}
$\rho_K$ is a well-defined representation of $B_n^+$.
\end{theorem}

\begin{proof}
Our definition only keeps track of the number of balls
in each collection of $K$ parallel lanes.
We must check that it is not necessary to know
precisely which lanes they are in.

Consider the cabling of a single positive crossing.
Suppose we bowl $a$ balls into
the upper $K$ lanes of a cabled crossing,
and $b$ into the lower $K$ lanes.
The empty lanes in the upper lanes will remain empty,
and the balls in the lower lanes will remain there.
The probability that exactly $c$ balls will fall
depends only on the number $a$ of occupied upper lanes
and the number $K - b$ of empty lower lanes.
\end{proof}

One motivation for studying $\rho_K$
would be to use it to compute the colored Jones polynomial of a knot.
Other apparently similar approaches to the colored Jones polynomial
have appeared in \cite{linwang} and \cite{armond}.
It would be interesting to know something about
the limiting behavior of $\rho_K$
if we set $q = e^{2 i \pi/K}$
and let $K$ go to infinity.
This may have some connection to
the Kashaev conjecture.

For all $0 \le c \le a \le N$ and $0 \le b \le N$,
let $f^a_b(c)$ denote the probability,
when $a$ balls enter the top lane of a crossing
and $b$ balls enter the bottom,
that $c$ balls fall from the upper lane to the lower lane.
We will compute a formula for $f^a_b(c)$.
First we define some notation.

If $k$ is a non-negative integer,
we define the quantum integer
$$[k] = \frac{1 - q^k}{1 - q}.$$
Note that we are using the definition
that involves only positive exponents of $q$,
not the definition
that is symmetric under mapping $q$ to $q^{-1}$.

The $q$-factorial is $[k]! = [k][k-1]\dots[1]$.
If $0 \le r \le k$,
the Gaussian binomial is
$$\qbinom{k}{r} = \frac{[k]!}{[r]![k-r]!}.$$
These have a combinatorial interpretation as follows.

An {\em inversion} of a permutation $\phi$ of $\{1,\dots,k\}$
is a pair such that $i < j$ and $\phi(i) > \phi(j)$.
The quantum factorial $[k]!$ is the sum over $\phi$
of $q$ to the power of the number of inversions of $\phi$.

An {\em inversion} of a sequence $(\epsilon_1, \dots, \epsilon_k)$
of ones and zeros
is a pair such that $i < j$, $\epsilon_i = 1$ and $\epsilon_j = 0$.
The Gaussian binomial $\qbinom{k}{r}$
is the sum over all such sequences that have $r$ ones
of $q$ to the power of the number of inversions.

Finally, let
$$\qbinom{\infty}{r} = \frac{1}{[r]!\,(1-q)^r}.$$
If $|q| < 1$ then this agrees with the obvious limit.
For general $q$,
we can just take it as a definition.

We now give a formula for $f^a_b(c)$.
Despite appearances,
it is a polynomial function of $q$.

\begin{theorem}
If $a,b,c$ are integers and $0 \le c \le a$
then
$$f^a_b(c) = \frac
  {\qbinom{a}{c} \qbinom{K - b}{c}}
  {\qbinom{\infty}{c}} \,
  q^{(a - c)(K - b - c)}.$$
\end{theorem}

\begin{proof}
Consider a crossing where
$K$ parallel lanes pass over
$K$ parallel lanes.
Now bowl $a$ balls into the upper collection of lanes
and $b$ into the lower.
We must compute the probability
that exactly $c$ balls will fall.

Fix a choice of $c$ of the upper $a$ balls,
a choice of $c$ of the lower $K - b$ empty lanes,
and a bijection $\phi$ from these balls to these empty lanes.
We compute the probability that
our chosen balls fall into our chosen lanes
according to $\phi$.

Some terminology will help us to stay organized.
The $K$ upper lanes consist of
our chosen $c$ {\em briefly-full} lanes,
$a - c$ {\em always-full} lanes,
and $K - a$ {\em irrelevant} lanes.
The $K$ lower lanes consist of
our chosen $c$ {\em briefly-empty} lanes,
$K - b - c$ {\em always-empty} lanes,
and $b$ {\em irrelevant} lanes.
The irrelevant lanes have no effect on the probability.

Consider the crossings where
an always-full lane passes over an always-empty lane.
At each such crossing,
a ball will pass over an empty lane,
contributing a factor of $q$.
Taken together,
these crossings contribute the term
  $$q^{(a - c)(K - b - c)}.$$

Consider the crossings where
an always-full lane passes over a briefly-empty lane.
Such a crossing will contribute a factor of $q$
if and only if
the briefly-empty lane is still empty,
having not yet met
its corresponding briefly-full lane.
The number of times this happens
is the number of pairs of lanes consisting of
a briefly-full lane to the left of an always-full lane.
Taken together,
these contribute the power of $q$ in the definition of
  $$\qbinom{a}{c}$$
(using the combinatorial definition,
with appropriate conventions).

Consider the crossings where
a briefly-full lane passes over an always-empty lane.
Such a crossing will contribute a factor of $q$
if and only if
the briefly-full lane is still full,
having not yet met
its corresponding briefly-empty lane.
The number of times this happens
is the number of pairs  of lanes consisting of
an always-empty lane to the left of a briefly-empty lane
Taken together,
these contribute the power of $q$ in the definition of
  $$\qbinom{K - b}{c}$$
(using the combinatorial definition,
with appropriate conventions).

Consider the crossings where
a ball falls from a briefly-full lane
to a briefly-empty lane.
Each such crossing contributes a factor of $(1 - q)$.
Taken together,
these contribute the term
  $$(1 - q)^c.$$

Finally,
consider the crossings where
a briefly-full lane passes over a briefly-empty lane
but no ball falls there.
This will contribute a factor of $q$
if and only if the briefly-full lane is still full
and the briefly-empty lane is still empty.
The number of times this happens
is the number of pairs of briefly-full lanes $i$ and $j$
such that $i$ is to the left of $j$
and $\phi(i)$ is to the left of $\phi(j)$.
This contributes the power of $q$ in the definition of
  $$[c]!$$
(using the combinatorial definition,
with appropriate conventions).

Now multiply the above contributions,
and sum over all possible choices of $c$ balls,
$c$ empty lanes,
and $\phi$.
This gives the desired formula for $f^a_b(c)$.
\end{proof}



\begin{thebibliography}{1}

\bibitem{armond}
Cody~W. Armond.
\newblock Walks along braids and the colored {J}ones polynomial.
\newblock {\em J. Knot Theory Ramifications}, 23(2):1450007, 15, 2014.

\bibitem{dipperjames}
Richard Dipper and Gordon James.
\newblock Representations of {H}ecke algebras of general linear groups.
\newblock {\em Proc. London Math. Soc. (3)}, 52(1):20--52, 1986.

\bibitem{jones}
V.~F.~R. Jones.
\newblock Hecke algebra representations of braid groups and link polynomials.
\newblock {\em Ann. of Math. (2)}, 126(2):335--388, 1987.

\bibitem{klw}
Paul Kirk, Charles Livingston, and Zhenghan Wang.
\newblock The {G}assner representation for string links.
\newblock {\em Commun. Contemp. Math.}, 3(1):87--136, 2001.

\bibitem{lintianwang}
Xiao-Song Lin, Feng Tian, and Zhenghan Wang.
\newblock Burau representation and random walks on string links.
\newblock {\em Pacific J. Math.}, 182(2):289--302, 1998.

\bibitem{linwang}
Xiao-Song Lin and Zhenghan Wang.
\newblock Random walk on knot diagrams, colored {J}ones polynomial and
  {I}hara-{S}elberg zeta function.
\newblock In {\em Knots, braids, and mapping class groups---papers dedicated to
  {J}oan {S}. {B}irman ({N}ew {Y}ork, 1998)}, volume~24 of {\em AMS/IP Stud.
  Adv. Math.}, pages 107--121. Amer. Math. Soc., Providence, RI, 2001.

\bibitem{reshetikhinturaev}
N.~Yu. Reshetikhin and V.~G. Turaev.
\newblock Ribbon graphs and their invariants derived from quantum groups.
\newblock {\em Comm. Math. Phys.}, 127(1):1--26, 1990.

\end{thebibliography}
\end{document}